\documentclass{amsart}
\usepackage{graphicx}
\usepackage[top=1.5in,bottom=1.5in,left=1.5in,right=1.5in]{geometry}
\usepackage{leftidx}
\newtheorem{theorem}{Theorem}[section]
\newtheorem*{theorem A}{Theorem A}
\newtheorem*{theorem B}{N\"olker's Theorem}
\newtheorem{lemma}{Lemma}[section]

\newtheorem{proposition}{Proposition}[section]

\newtheorem{corollary}{Corollary}[section]

\theoremstyle{remark}
\newtheorem{remark}{Remark}[section]
\theoremstyle{remark}

\theoremstyle{definition}

\numberwithin{equation}{section}
\def\({\left ( }
\def\){\right )}
\def\<{\left < }
\def\>{\right >}

\begin{document}

\title{ Ricci solitons in almost $f$-cosymplectic manifolds }

\author{Xiaomin Chen}
\address{College of  Science, China University of Petroleum-Beijing, Beijing, 102249, China}
\email{xmchen@cup.edu.cn}

\thanks{The author is supported by the Science Foundation of China
University of Petroleum-Beijing(No.2462015YQ0604) and partially
by  the Personnel Training and Academic
Development Fund (2462015QZDX02).}


\begin{abstract}
In this article we study an almost $f$-cosymplectic manifold admitting a Ricci soliton. We first prove that there do not exist Ricci solitons on an almost cosymplectic $(\kappa,\mu)$-manifold. Further, we consider an almost
$f$-cosymplectic manifold admitting a Ricci soliton whose potential vector field is the Reeb vector field and show that
a three dimesional almost $f$-cosymplectic is a cosymplectic manifold. Finally we classify a three dimensional $\eta$-Einstein almost $f$-cosymplectic manifold admitting a Ricci soliton.
\end{abstract}



\maketitle

\section{Introduction}

A \emph{Ricci soliton} is a Riemannian metric defined on manifold $M$ such that
\begin{equation}\label{1.1}
 \frac{1}{2}\mathcal{L}_V g+Ric-\lambda g=0,
\end{equation}
where $V$ and $\lambda$ are the potential vector field and some constant on $M$, respectively.
Moreover, the Ricci soliton is called
 \emph{shrinking, steady }and \emph{expanding} according as $\lambda$ is positive, zero and negative respectively.
  The Ricci solitons are of interest to physicists as well and are
known as quasi Einstein metrics in the physics literature
\cite{F}. Compact Ricci solitons are the
fixed point of the Ricci flow:$\frac{\partial}{\partial t}g=-2Ric$,
projected from the space of metrics onto its quotient modulo diffeomorphisms and scalings, and often
arise as blow-up limits for the Ricci flows on compact manifolds. The study on the Ricci solitons
has a long history and a lot of achievements were acquired, see \cite{FG,MS,PW}etc.
On the other hand,  the normal almost contact manifolds admitting Ricci solitons were also been studied by many researchers (see \cite{C2,G,GS,GS2}).

Recently, we note that the three dimensional almost Kenmotsu manifolds admitting Ricci solitons were considered (\cite{WL,WDL}) and Cho (\cite{C1}) gave the classification of an almost cosymplectic manifold admitting a Ricci soliton whose potential vector field is the Reeb vector field.
Here the \emph{almost cosymplectic manifold}, defined by Goldberg and Yano \cite{GY}, was an almost contact manifold whose 1-form $\eta$ and fundamental 2-form $\omega$ are closed, and the {\it almost Kenmotsu manifold} is an almost contact manifold satisfying $d\eta=0$ and $d\omega=2\eta\wedge\omega$.
    Based on this Kim and Pak \cite{KP} introduced the concept of \emph{almost $\alpha$-cosymplectic manifold}, i.e., an almost contact manifold satisfying
$d\eta=0$ and $d\omega=2\alpha\eta\wedge\omega$ for any real number $\alpha$. In particular, if $\alpha$ is non-zero it is said to be an \emph{almost $\alpha$-Kenmotsu} manifold.  Later Aktan et al.\cite{AYM} defined an \emph{almost $f$-cosymplectic manifold} $M$ by generalizing the real number $\alpha$ to a smooth function $f$ on $M$, i.e., an almost contact manifold  satisfies $d\omega=2f\eta\wedge\omega$ and $d\eta=0$ for a smooth function $f$ satisfying $df\wedge\eta=0$.  Clearly, an almost $f$-cosymplectic manifold is an almost cosymplectic manifold under the condition that $f=0$ and an almost $\alpha$-Kenmotsu manifold if $f$ is constant$(\neq0)$.
In particular, if $f=1$ then $M$ is an almost Kenmotsu manifold.

On the other hand, we observe that a remarkable class of contact metric manifold is $(\kappa,\mu)$-space whose curvature tensor satisfies
\begin{equation}\label{1.2}
  R(X,Y)\xi=\kappa(\eta(Y)X-\eta(X)Y)+\mu(\eta(Y)hX-\eta(X)hY)
\end{equation}
for any vector fields $X,Y$, where $\kappa$ and $\mu$ are constants and $h:=\frac{1}{2}\mathcal{L}_\xi\phi$ is a self-dual operator.
 In fact Sasakian manifolds
are special $(\kappa,\mu)$-spaces with $\kappa=1$ and $h=0$. An \emph{almost cosymplectic $(\kappa,\mu)$-manifold} is
an almost cosymplectic manifold with curvature tensor  satisfying \eqref{1.2}.
 Endo proved that if $\kappa\neq0$ any almost cosymplectic $(\kappa,\mu)$-manifolds are not cosymplectic (\cite{E2}). Furthermore, since $\kappa\phi^2=h^2$, $\kappa\leq0$ and the equality holds if and only if the almost cosymplectic $(\kappa,\mu)$-manifolds are cosymplectic.
Notice that Wang proved the non-existence of gradient Ricci solitons in almost cosymplectic $(\kappa,\mu)$-manifolds(see \cite{W}).

In this paper we first obtain an non-existence of a Ricci soliton in almost cosymplectic $(\kappa,\mu)$-manifolds, namely

\begin{theorem}\label{T1}
There do not exist Ricci solitons on almost cosymplectic $(\kappa,\mu)$-manifolds  with $\kappa<0$ and $-\frac{1}{2}<\mu<0$.
\end{theorem}

 Next we consider a three-dimensional almost $f$-cosymplectic manifold admitting a Ricci soliton whose potential vector field is the Reeb vector field $\xi$, and prove the following theorem.
\begin{theorem}\label{T2}
A three-dimensional almost $f$-cosymplectic manifold $M$ admits a Ricci soliton whose potential vector field is $\xi$
if and only if $\xi$ is Killing and $M$ is Ricci flat.
\end{theorem}

As it is well known that a $(2n+1)$-dimensional almost contact manifold $(M,\phi,\xi,\eta,g)$
is said to be \emph{$\eta$-Einstein} if its Ricci tensor satisfies
\begin{equation}\label{1.3}
Ric=ag+b\eta\otimes\eta,
\end{equation}
where $a$ and $b$ are smooth functions.
For a three-dimensional $\eta$-Einstein almost $f$-cosymplectic manifold $M$ with a Ricci soliton
we prove the following result:
\begin{theorem}\label{T3}
Let $(M,\phi,\eta,\xi,g)$ be a three-dimensional $\eta$-Einstein almost $f$-cosymplectic manifold admitting a Ricci soliton. Then
either $M$ is an $\alpha$-cosymplectic manifold, or
   $M$ is an Einstein manifold of constant sectional curvature $\frac{\lambda}{2} $ with $\lambda=-2\xi(f)-2f^2$.
\end{theorem}
\begin{remark}
Our theorem extends the Wang and Liu's result \cite{WL}. In fact, when $f=1$ then $ a=-2$ in view of Proposition \ref{p2} in Section 5.
Thus it follows from \eqref{2.8} that $\mathrm{trace}(h^2)=0$, i.e., $h=0$. So $M$ is also a Kenmotsu manifold of sectional curvature $-1$.
\end{remark}
In order to prove these conclusions, in Section 2 we recall some basic concepts and formulas. The proofs of theorems will be given in Section 3,
Section 4 and Section 5, respectively.

\section{Some basic concepts and related results}

In this section we will recall some basic concepts and equations. Let $M^{2n+1}$ be a $(2n+1)$-dimensional Riemannian manifold.
An \emph{almost contact structure} on $M$ is a triple $(\phi,\xi,\eta)$, where $\phi$ is a
$(1,1)$-tensor field, $\xi$ a unit vector field, $\eta$ a one-form dual to $\xi$ satisfying
\begin{equation}\label{2.1}
\phi^2=-I+\eta\otimes\xi,\,\eta\circ \phi=0,\,\phi\circ\xi=0.
\end{equation}

A smooth manifold with such a structure is called an \emph{almost contact manifold}.
It is well-known that there exists a Riemannian metric $g$ such that
\begin{equation}\label{2.2}
g(\phi X,\phi Y)=g(X,Y)-\eta(X)\eta(Y),
\end{equation} for any $X,Y\in\mathfrak{X}(M)$. It is easy to get from \eqref{2.1} and \eqref{2.2} that
\begin{equation}\label{2.3}
g(\phi X,Y)=-g(X,\phi Y),\quad g(X,\xi)=\eta(X).
\end{equation}
An almost contact structure $(\phi,\xi,\eta)$ is said
to be \emph{normal} if the corresponding complex structure $J$ on $M\times\mathbb{R}$ is integrable.

Denote by $\omega$ the fundamental 2-form on $M$ defined by $\omega(X,Y):=g(\phi X,Y)$ for all $X,Y\in\mathfrak{X}(M)$.
An {\it almost $\alpha$-cosymplectic manifold} (\cite{KP,OAM}) is an almost contact metric manifold $(M,\phi,\xi,\eta,g)$ such that the fundamental form $\omega$ and 1-form $\eta$ satisfy $d\eta=0$ and $d\omega=2\alpha\eta\wedge\omega,$ where $\alpha$ is a real number. In particular, $M$ is an {\it almost cosymplectic manifold} if $\alpha=0$
 and an {\it almost Kenmotsu manifold} if $\alpha=1$.
In \cite{AYM}, a class of more general almost contact manifolds was defined by generalizing the real number $\alpha$ to a smooth function $f$.  More precisely, an almost contact metric manifold is called an \emph{almost $f$-cosymplectic manifold} if $d\eta=0$ and $d\omega=2f\eta\wedge\omega$ are satisfied,
where $f$ is a smooth function with $df\wedge\eta=0.$ In addition, a normal almost $f$-cosymplectic manifold is
 said to be an \emph{$f$-cosymplectic manifold.}

Let $M$ be an almost $f$-cosymplectic manifold, we recall that there is an operator
$h=\frac{1}{2}\mathcal{L}_\xi\phi$ which is a self-dual operator.  The Levi-Civita connection
is given by (see \cite{AYM})
\begin{equation}\label{2.4}
  2g((\nabla_X\phi)Y,Z)=2fg(g(\phi X,Y)\xi-\eta(Y)\phi X,Z)+g(N(Y,Z),\phi X)
\end{equation}
for arbitrary vector fields $X,Y$, where $N$ is the Nijenhuis torsion of $M$.  Then by a simple calculation, we have
\begin{equation*}
\mathrm{trace}(h)=0,\quad h\xi=0,\quad\phi h=-h\phi,\quad g(hX,Y)=g(X,hY),\quad\forall X,Y\in\mathfrak{X}(M).
\end{equation*}

Write $AX:=\nabla_X\xi$ for any vector field $X$. Thus $A$ is a $(1,1)$-tensor of $M$. Using \eqref{2.4}, a straightforward calculation gives
\begin{equation}\label{2.5}
AX= -f\phi^2X-\phi hX
\end{equation}
and $\nabla_\xi\phi=0$. By \eqref{2.3}, it is obvious that $A\xi=0$ and $A$ is symmetric with respect to metric $g$, i.e., $g(AX,Y)=g(X,AY)$ for all $X,Y\in\mathfrak{X}(M)$. We denote by $R$ and $\mathrm{Ric}$ the Riemannian curvature tensor and Ricci tensor, respectively. For an almost $f$-cosymplectic manifold $(M^{2n+1},\phi,\xi,\eta,g)$ the following equations were proved(\cite{AYM}):
\begin{align}
&R(X,\xi)\xi-\phi R(\phi X,\xi)\xi=2[(\xi(f)+f^2)\phi^2X-h^2X]\label{2.6},\\
&(\nabla_\xi h)X= -\phi R(X,\xi)\xi-[\xi(f)+f^2]\phi X-2fhX-\phi h^2X,\label{2.7} \\
&Ric(\xi,\xi)=-2n(\xi(f)+f^2)-\mathrm{trace}(h^2),\label{2.8}\\
&\mathrm{trace}(\phi h)=0,\label{2.9}\\
 &R(X,\xi)\xi=[\xi(f)+f^2]\phi^2X+2f\phi hX-h^2X+\phi(\nabla_\xi h)X\label{2.10},
 \end{align}
for any vector fields $X,Y$ on $M$.

\section{Proof of Theorem \ref{T1} }
In this section we suppose that $(M,\phi,\xi,\eta,g)$ is an almost cosymplectic $(\kappa,\mu)$-manifold, i.e., the
curvature tensor satisfies \eqref{1.2}.
In the following we always suppose
$\kappa<0$. The following relations are provided(see \cite[Eq.(3.22) and Eq.(3.23)]{MNY}):
\begin{align}
  Q=&2n\kappa\eta\otimes\xi+\mu h,\label{3.1}\\
  h^2=&\kappa\phi^2\label{3.2},
\end{align}
where $Q$ is the Ricci operator defined by $\mathrm{Ric}(X,Y)=g(QX,Y)$ for any vectors $X,Y$. In particular,  $Q\xi=2n\kappa\xi$ because of $h\xi=0$.

In view of \eqref{3.1} and the Ricci soliton equation \eqref{1.1}, we obtain
\begin{equation}\label{3.3}
  (\mathcal{L}_V g)(Y,Z)=2\lambda g(Y,Z)-2\mu g(hY,Z)-4n\kappa\eta(Y)\eta(Z)
\end{equation}
for any vectors $Y,Z$. Since $\kappa,\mu$ are two real numbers and $\nabla_X\xi=AX$, differentiating \eqref{3.3} along any vector field $X$ provides
\begin{align}\label{3.4}
   (\nabla_X\mathcal{L}_V g)(Y,Z)=&\nabla_X((\mathcal{L}_V g)(Y,Z))-\mathcal{L}_V g(\nabla_XY,Z)-\mathcal{L}_V g(Y,\nabla_XZ)\\
   =&-2\mu g((\nabla_Xh)Y,Z)-4n\kappa\nabla_X(\eta(Y))\eta(Z)-4n\kappa\eta(Y)\nabla_X(\eta(Z))\nonumber\\
   &+4n\kappa\eta(\nabla_XY)\eta(Z)+4n\kappa\eta(Y)\eta(\nabla_XZ)\nonumber\\
   =&-2\mu g((\nabla_Xh)Y,Z)-4n\kappa g(Y,AX)\eta(Z)-4n\kappa\eta(Y)g(Z,AX).\nonumber
\end{align}

Moreover, making use of the commutation formula (see \cite{Y}):
\begin{equation*}
\begin{aligned}
    (&\mathcal{L}_V\nabla_X g-\nabla_X\mathcal{L}_V g-\nabla_{[V,X]}g)(Y,Z)= \\
   & -g((\mathcal{L}_V\nabla)(X,Y),Z)-g((\mathcal{L}_V\nabla)(X,Z),Y),
\end{aligned}
 \end{equation*}
we derive
\begin{align}\label{3.5}
   (\nabla_X\mathcal{L}_V g)(Y,Z)=g((\mathcal{L}_V\nabla)(X,Y),Z)+g((\mathcal{L}_V\nabla)(X,Z),Y).
\end{align}
It follows from  \eqref{3.4} and \eqref{3.5} that
\begin{align}\label{3.6}
g((\mathcal{L}_V\nabla)(Y,Z),X)=&\frac{1}{2}\Big\{(\nabla_Z\mathcal{L}_V g)(Y,X)+(\nabla_Y\mathcal{L}_V g)(Z,X)-(\nabla_X\mathcal{L}_V g)(Y,Z)\Big\} \\
    =&-\mu \Big\{g((\nabla_Z h)Y,X)+g((\nabla_Y h)Z,X)-g((\nabla_X h)Y,Z)\nonumber\\
    &-4n\kappa g(Y,AZ)\eta(X)\Big\}.\nonumber
\end{align}
Hence for any vector $Y$,
\begin{equation}\label{3.7}
  (\mathcal{L}_V\nabla)(Y,\xi)=-\mu \Big\{(\nabla_\xi h)Y+2\kappa\phi Y\Big\}
\end{equation}
by using \eqref{3.2} and \eqref{2.5}.
Lie differentiating \eqref{3.7} along $V$ and making use of the identity(\cite{Y}):
\begin{equation}\label{3.8}
(\mathcal{L}_VR)(X,Y)Z =(\nabla_X\mathcal{L}_V\nabla)(Y,Z)-(\nabla_Y\mathcal{L}_V\nabla)(X,Z),
\end{equation}
 we obtain
\begin{align}\label{3.9}
(\mathcal{L}_VR)(X,\xi)\xi=&(\nabla_X\mathcal{L}_V\nabla)(\xi,\xi)-(\nabla_\xi\mathcal{L}_V\nabla)(X,\xi)\\
=&-2(\mathcal{L}_V\nabla)(\nabla_X\xi,\xi)-(\nabla_\xi\mathcal{L}_V\nabla)(X,\xi)\nonumber\\
=&2\mu \Big\{(\nabla_\xi h)\nabla_X\xi+2\kappa\phi\nabla_X\xi\Big\}+\mu\{(\nabla_\xi\nabla_\xi h)X)\}\nonumber\\
=&2\mu \Big\{\phi(\nabla_\xi h)hX+2\kappa hX\Big\}+\mu\{(\nabla_\xi\nabla_\xi h)X)\}\nonumber\\
=&\mu \Big\{2\kappa\mu\phi^2X+(4\kappa+\mu) hX\Big\}.\nonumber
 \end{align}
Here we have used the following fact: By \eqref{2.7}, \eqref{3.2} and \eqref{2.1}, we get $\nabla_\xi h=-\mu\phi h,$ thus
\begin{align*}
  \nabla_\xi\nabla_\xi h+2\phi(\nabla_\xi h)h=\mu h+2\kappa\mu\phi^2.
\end{align*}

Lie differentiating $R(X,\xi)\xi=-\kappa\phi^2X+\mu hX$ (a consequence of Eq.\eqref{1.2}) and recalling \eqref{3.9}, we have
 \begin{align}\label{3.10}
   \mu& \Big\{2\kappa\mu\phi^2X+(4\kappa+\mu) hX\Big\}+R(X,\mathcal{L}_V\xi)\xi+R(X,\xi)\mathcal{L}_V\xi \\
   = & -\kappa(\mathcal{L}_V\eta)(X)\xi-\kappa\eta(X)\mathcal{L}_V\xi+\mu (\mathcal{L}_Vh)X.\nonumber
 \end{align}
 Using \eqref{1.2} again gives
 \begin{align}\label{3.11}
   R(X,\xi)\mathcal{L}_V\xi & =R(X,\mathcal{L}_V\xi)\xi-\Big[\kappa \Big(g(X,\mathcal{L}_V\xi)\xi-\eta(X)\mathcal{L}_V\xi\Big) \\
    & +\mu\Big(g(hX,\mathcal{L}_V\xi)\xi-\eta(X)h\mathcal{L}_V\xi\Big)\Big].\nonumber
 \end{align}
 On the other hand, it follows from \eqref{3.3} that
 \begin{equation}\label{3.12}
 (\mathcal{L}_V\eta)(X)=(\mathcal{L}_Vg)(X,\xi)+g(X,\mathcal{L}_V\xi)=(2\lambda-4n\kappa)\eta(X)+g(X,\mathcal{L}_V\xi).
 \end{equation}
Now substituting \eqref{3.11} and \eqref{3.12} into \eqref{3.10} and using \eqref{1.2}, we obtain
 \begin{align}\label{3.13}
   \mu& \Big\{2\kappa\mu\phi^2X+(4\kappa+\mu) hX\Big\}+2\kappa\eta(\mathcal{L}_V\xi)X+2\mu\eta(\mathcal{L}_V\xi)hX \\
   &-\kappa g(X,\mathcal{L}_V\xi)\xi-\kappa\eta(X)\mathcal{L}_V\xi-\mu g(hX,\mathcal{L}_V\xi)\xi-\mu\eta(X)h\mathcal{L}_V\xi\nonumber\\
   = & -\kappa\Big[(2\lambda-4n\kappa)\eta(X)+g(X,\mathcal{L}_V\xi)\Big]\xi-\kappa\eta(X)\mathcal{L}_V\xi+\mu (\mathcal{L}_Vh)X.\nonumber
 \end{align}

Furthermore,  notice that the Ricci tensor equation \eqref{3.1} implies the scalar curvature
$r=2n\kappa$ and recall the following integrability formula (see  \cite[Eq.(5)]{S}):
\begin{equation*}
  \mathcal{L}_Vr=-\Delta r-2\lambda r+2||Q||^2
\end{equation*}
 for a Ricci soliton. By \eqref{3.1}, \eqref{3.2} and the foregoing formula we thus obtain
 $\lambda=2n\kappa-\mu^2$.  Also, it follows from \eqref{3.3} that $g(\mathcal{L}_V\xi,\xi)=2n\kappa-\lambda$. Therefore $g(\mathcal{L}_V\xi,\xi)=\mu^2$.  Hence 
Equation \eqref{3.13} becomes
  \begin{align}\label{3.14}
   \mu& (4\kappa+\mu) hX+2\mu^3hX -\kappa\eta(X)\mathcal{L}_V\xi-\mu g(hX,\mathcal{L}_V\xi)\xi-\mu\eta(X)h\mathcal{L}_V\xi\nonumber\\
   = & -\kappa\eta(X)\mathcal{L}_V\xi+\mu (\mathcal{L}_Vh)X.\nonumber
 \end{align}
Letting $X\in\mathfrak{D}$ be an eigenvector of $h$ in the above formula and taking the inner product with $X$, we find
\begin{equation*}
   \mu (4\kappa+\mu+2\mu^2)=0
\end{equation*}
If $\mu\neq0$ then $4\kappa+\mu+2\mu^2=0$, which implies $\mu<-\frac{1}{2}$ or $\mu>0.$

If $\mu=0$. From \eqref{3.4}, \eqref{3.5} and \eqref{3.6}, we have
 \begin{equation*}
   g(Y,AX)\eta(Z)+g(Z,AX)\eta(Y)=0
\end{equation*}
since $\kappa<0$. Now putting $Z=\xi$ gives $g(Y,AX)=0$ for any vector fields $X,Y$ because $A\xi=0$. That means that $AX=0$ for any vector field $X$. From \eqref{2.5} with $f=0$, we get $h=0$. Clearly, it is impossible. Therefore we complete the proof.

\section{Proof of Theorem \ref{T2}}
In this section we assume that $M$ is a three dimensional almost $f$-cosymplectic
manifold and the potential vector field $V$ is the Reeb vector field. Before giving the proof, we need to prove the following lemma.
 \begin{lemma}\label{L1}
For any almost $f$-cosymplectic manifold the following formula holds:
\begin{align*}
(\mathcal{L}_\xi R)(X,\xi)\xi=&2\xi(f)\phi hX-2[f(\nabla_\xi h)\phi X+(\nabla_\xi h)hX]\nonumber\\
&+[\xi(\xi(f))+2f\xi(f)]\phi^2X+\phi(\nabla_\xi\nabla_\xi h)X.\nonumber
 \end{align*}
\end{lemma}
\begin{proof}
 Obviously,  $\mathcal{L}_\xi\eta=0$ because $A\xi=0$.  Notice that for any vector fields $X,Y,Z$ the 1-form $\eta$  satisfies the following relation(\cite{Y}):
\begin{equation}\label{4.1}
  -\eta((\mathcal{L}_X\nabla)(Y,Z))=(\mathcal{L}_X(\nabla_Y\eta)-\nabla_Y(\mathcal{L}_X\eta)-\nabla_{[X,Y]}\eta)(Z).
\end{equation}
Putting $X=\xi$ and using \eqref{2.5} yields
\begin{align}\label{4.2}
 -\eta((\mathcal{L}_\xi\nabla)(Y,Z))=& (\mathcal{L}_\xi(\nabla_Y\eta)-\nabla_Y(\mathcal{L}_\xi\eta)-\nabla_{[\xi,Y]}\eta)(Z)\\
   = &\nabla_\xi((\nabla_Y\eta)(Z))-(\nabla_Y\eta)(\mathcal{L}_\xi Z)-g(A([\xi,Y]),Z)\nonumber\\
   =&\nabla_\xi g(AY,Z)-g(AY,[\xi,Z])-g(A([\xi,Y]),Z)\nonumber\\
   =&g((\nabla_\xi A)Y,Z)+2g(AY,AZ).\nonumber
\end{align}
In view of \eqref{3.5}, we obtain from \eqref{4.2} that
\begin{align*}
  g((\mathcal{L}_\xi\nabla)(X,\xi),Y)
  =&(\nabla_X\mathcal{L}_\xi g)(Y,\xi)-g((\mathcal{L}_\xi\nabla)(X,Y),\xi)\\
  =&-2g(AX,AY)+g((\nabla_\xi A)X,Y)+2g(AX,AY)\\
   =&\xi(f)g(\phi X,\phi Y)+g((\nabla_\xi h)X,\phi Y).
\end{align*}
That is,
\begin{equation*}
  (\mathcal{L}_\xi\nabla)(X,\xi)=-[\xi(f)]\phi^2 X-\phi(\nabla_\xi h)X.
\end{equation*}

Obviously, for any vector field $X$, we know $(\mathcal{L}_\xi\nabla)(X,\xi)=(\mathcal{L}_\xi\nabla)(\xi,X)$ from \eqref{3.6}.
Therefore we compute the Lie derivative of $R(X,\xi)\xi$ along $\xi$ as follows:
\begin{align*}\label{eq3.17}
(\mathcal{L}_\xi R)(X,\xi)\xi=&(\nabla_X\mathcal{L}_\xi\nabla)(\xi,\xi)-(\nabla_\xi\mathcal{L}_\xi\nabla)(X,\xi)\\
=&-(\mathcal{L}_\xi\nabla)(\nabla_X\xi,\xi)-(\mathcal{L}_\xi\nabla)(\xi,\nabla_X\xi)-(\nabla_\xi\mathcal{L}_\xi\nabla)(X,\xi)\nonumber\\
=&2[\xi(f)]\phi^2 AX+2\phi(\nabla_\xi h)AX-(\nabla_\xi\mathcal{L}_\xi\nabla)(X,\xi)\nonumber\\
=&2\xi(f)(f\phi^2X+\phi hX)+2[-f(\nabla_\xi h)\phi X-(\nabla_\xi h)hX]\nonumber\\
&+[\xi(\xi(f))]\phi^2X+\phi(\nabla_\xi\nabla_\xi h)X.\nonumber
\end{align*}
Here we used $\nabla_\xi\phi=0$ and $\phi\nabla_\xi h=-(\nabla_\xi h)\phi$ followed from $h\phi+\phi h=0$.
\end{proof}\bigskip

{\it Proof of Theorem \ref{T2}.} Now we suppose that the potential vector $V=\xi$
in the Ricci equation \eqref{1.1}. Then for any $X\in\mathfrak{X}(M)$,
\begin{equation}\label{4.3}
  -f\phi^2X-\phi hX+QX=\lambda X.
\end{equation}
Putting $X=\xi$ in \eqref{4.3}, we have
\begin{equation}\label{4.4}
 Q\xi=\lambda\xi.
\end{equation}
 Moreover, the above formula together \eqref{2.8} with $n=1$ leads to
\begin{equation}\label{4.5}
  \mathrm{trace}(h^2)=-\lambda-2f^2-2\xi(f).
\end{equation}

Since the curvature tensor of a 3-dimension Riemannian manifold is given by
\begin{align}\label{4.6}
  R(X,Y)Z=&g(Y,Z)QX-g(X,Z)QY+g(QY,Z)X-g(QX,Z)Y\\
           &-\frac{r}{2}\{g(Y,Z)X-g(X,Z)Y\},\nonumber
  \end{align}
where $r$ denotes the scalar curvature. Putting $Y=Z=\xi$ in \eqref{4.6} and
applying \eqref{4.3} and \eqref{4.4}, we obtain
\begin{align}\label{4.7}
  R(X,\xi)\xi=\Big(\frac{r}{2}+f-2\lambda\Big)\phi^2X+\phi hX.
\end{align}
Contracting the above formula over $X$ leads to $Ric(\xi,\xi)=-r-2f+4\lambda$, which follows from \eqref{4.4} that
\begin{equation}\label{4.8}
  r+2f=3\lambda.
\end{equation}

Taking the Lie derivative of \eqref{4.7} along $\xi$ and using \eqref{4.8}, we obtain
\begin{align}\label{4.9}
  (\mathcal{L}_\xi R)(X,\xi)\xi=2h^2X+\phi(\mathcal{L}_\xi h)X
\end{align}
since $\mathcal{L}_\xi\phi^2=2\phi h+2h\phi=0$ and $h=\frac{1}{2}\mathcal{L}_\xi\phi$.
By Lemma \ref{L1} and \eqref{4.9}, we get
\begin{align}\label{4.10}
2&\xi(f)\phi hX-2[f(\nabla_\xi h)\phi X+(\nabla_\xi h)hX]\\
&+[\xi(\xi(f))+2f\xi(f)]\phi^2X+\phi(\nabla_\xi\nabla_\xi h)X\nonumber\\
=&2h^2X+\phi(\mathcal{L}_\xi h)X.\nonumber
  \end{align}
By virtue of \eqref{2.7} and \eqref{4.7}, we have
\begin{equation}\label{4.11}
  (\nabla_\xi h)X=\Big( -\frac{\lambda}{2}-\xi(f)-f^2\Big)\phi X+(1-2f)hX-\phi h^2X.
\end{equation}
Making use of the above equation we further compute
\begin{align}\label{4.12}
  \phi(\nabla_\xi\nabla_\xi h)X=&\Big(-\xi\xi(f)-2f\xi(f)\Big)\phi^2 X-2\xi(f)\phi hX+(1-2f)\phi(\nabla_\xi h)X\\
                            &+(\nabla_\xi h)hX+ h(\nabla_\xi h)X.\nonumber
  \end{align}
As well as via \eqref{2.5} we get
\begin{align}\label{4.13}
  \phi(\mathcal{L}_\xi h)X=&\phi\mathcal{L}_\xi(hX)-\phi h([\xi,X])\\
  =&\phi(\nabla_\xi h)X+\phi (hA-Ah)X\nonumber\\
  =&\phi(\nabla_\xi h)X-2h^2X.\nonumber
\end{align}

Substituting \eqref{4.12} and \eqref{4.13} into \eqref{4.10}, we derive
\begin{align}\label{4.14}
   -(\nabla_\xi h)hX+h(\nabla_\xi h)X=0.
\end{align}
Further, applying \eqref{4.11} in the formula \eqref{4.14}, we get
\begin{align*}
  \Big(-\lambda-2f^2-2\xi(f)\Big)hX-h^3X=0.
\end{align*}
Write $\beta:=-\lambda-2f^2-2\xi(f)$, then the above equation is rewritten as $h^3X=\beta hX$ for every vector $X$.
Denote by $e_i$ and $\lambda_i$ the eigenvectors and the corresponding eigenvalues for $i=1,2,3$, respectively. If $h\neq0$ then there is a nonzero eigenvalue $\lambda_1\neq0$ satifying $\lambda_1^3=\beta \lambda_1$, i.e., $\lambda_1^2=\beta$. Since $\mathrm{trace}(h)=0$ and $\mathrm{trace}(h^2)=\beta$ by \eqref{4.4}, we have
$\sum_{i=1}^3\lambda_i=0$ and $\sum_{i=1}^3\lambda_i^2=\beta$. This shows $\lambda_2=\lambda_3=0,$ which further leads to $\lambda_1=0$. It is a contradiction.
Thus we have
\begin{equation}\label{4.15}
\beta=-(\lambda+2f^2+2\xi(f))=0\quad\hbox{and}\quad h=0.
\end{equation}

On the other hand, taking the covariant differentiation of $Q\xi=\lambda\xi$ (see \eqref{4.4}) along arbitrary vector field $X$ and
using \eqref{2.5}, one can easily deduce
\begin{equation*}
  (\nabla_XQ)\xi+Q(-f\phi^2X-\phi hX)=-\lambda(f\phi^2X+\phi hX).
\end{equation*}
Contracting this equation over $X$ and using \eqref{4.3}, we derive
\begin{equation*}
  \frac{1}{2}\xi(r)-2f^2-\mathrm{trace}(h^2)=0.
\end{equation*}
By virtue of \eqref{4.4} and \eqref{4.8}, the foregoing equation yields
\begin{equation}\label{4.16}
\xi(f)+\lambda=0.
\end{equation}
This shows that $\xi(f)$ is constant, then it infers from \eqref{4.15} that $f=0$ and $\lambda=0$, that means that $M$ is cosymplectic.
Therefore it follows from \eqref{2.5} that $(\mathcal{L}_\xi g)(X,Y)=2g(AX,Y)=0$ for any vectors $X,Y$.
By \eqref{4.3}, it is obvious that $Q=0$. Thus we complete the proof of Theorem \ref{T2}.\\

 By the proof of Theorem \ref{T2}, the following result is clear.
\begin{corollary}
A three-dimensional almost $\alpha$-Kenmotsu manifold $(M,\phi,\eta,\xi,g)$ does not admit a Ricci
 soliton with potential vector field being $\xi$.
\end{corollary}

\section{Proof of Theorem \ref{T3}}

In this section we assume that $M$ is a three dimensional $\eta$-Einstein almost
$f$-cosymplectic manifold, i.e., the Ricci tensor $Ric=ag+b\eta\otimes\eta$. We  first prove the following proposition.
\begin{proposition}\label{p2}
Let $M$ be a three dimensional $\eta$-Einstein almost $f$-cosymplectic manifold. Then  the following relation is satisfied:
\begin{equation*}
2\xi(f)+2f^2+(a+b)=0.
\end{equation*}
\end{proposition}
\begin{proof}
 Since $M$ is $\eta$-Einstein, we know that $Q\xi=(a+b)\xi$ and the scalar curvature $r=3a+b$.
Hence it follows from \eqref{4.6} that
\begin{align}\label{5.1}
R(X,\xi)\xi=&-\frac{a+b}{2}\phi^2X.
\end{align}

By Lie differentiating \eqref{5.1}, we derive
\begin{align}\label{5.2}
  (\mathcal{L}_\xi R)(X,\xi)\xi=-\frac{\xi(a+b)}{2}\phi^2X.
\end{align}
Also, it follows from \eqref{5.1} and \eqref{2.7} that
\begin{align}\label{5.3}
  (\nabla_\xi h)X=&-[\xi(f)+f^2+\frac{1}{2}(a+b)]\phi X-2fhX-\phi h^2X.
\end{align}
Moreover, we get
\begin{align}\label{5.4}
 \phi(\nabla_\xi\nabla_\xi h)X=&-\Big[\xi(\xi(f))+2f\xi(f)+\frac{\xi(a+b)}{2}\Big]\phi^2 X\\
  & -2\xi(f)\phi hX-2f\phi(\nabla_\xi h)X+(\nabla_\xi h)hX+ h(\nabla_\xi h)X.\nonumber
\end{align}
Therefore by using \eqref{5.3} and \eqref{5.4}, the formula of Lemma \ref{L1} becomes
\begin{align*}
(\mathcal{L}_\xi R)(X,\xi)\xi=&-\frac{\xi(a+b)}{2}\phi^2X-(\nabla_\xi h)hX+ h(\nabla_\xi h)X\nonumber\\
=&-\frac{\xi(a+b)}{2}\phi^2X+2\Big[\xi(f)+f^2+\frac{1}{2}(a+b)\Big]\phi hX+2\phi h^3X.
 \end{align*}
Combining this with \eqref{5.2} yields
\begin{equation}\label{5.5}
 \begin{aligned}
0=[\xi(f)+f^2+\frac{1}{2}(a+b)]hX+h^3X.
 \end{aligned}
 \end{equation}
As in the proof of \eqref{4.15}, the formula \eqref{5.5} yields the assertion.
\end{proof}

{\it Proof of Theorem \ref{T3}.} In view of Proposition \ref{p2},  we know $a+b=-2\xi(f)-2f^2$.
Because $(\mathcal{L}_V g)(Y,Z)=2\lambda g(Y,Z)-2g(QY,Z)$ for any vector fields $Y,Z$, we compute
\begin{align*}
  &(\nabla_X\mathcal{L}_V g)(Y,Z)=-2g((\nabla_XQ)Y,Z)\\
  =&-2g(X(a)Y+X(b)\eta(Y)\xi+b g(AX,Y)\xi+b\eta(Y)AX,Z).
\end{align*}
Hence
\begin{align*}
g((\mathcal{L}_V\nabla)(Y,Z),X)=&\frac{1}{2}\Big\{(\nabla_Z\mathcal{L}_V g)(Y,X)+(\nabla_Y\mathcal{L}_V g)(Z,X)-(\nabla_X\mathcal{L}_V g)(Y,Z)\Big\}\\
=&-\Big\{g(Z(a)Y+Z(b)\eta(Y)\xi+b g(AZ,Y)\xi+b\eta(Y)AZ,X)\\
&+g(Y(a)Z+Y(b)\eta(Z)\xi+b g(AY,Z)\xi+b\eta(Z)AY,X)\\
&-g(X(a)Y+X(b)\eta(Y)\xi+b g(AX,Y)\xi+b\eta(Y)AX,Z)\Big\}\\
=&-\Big\{Z(a)g(Y,X)+Z(b)\eta(Y)\eta(X)+Y(a)g(X,Z)+Y(b)\eta(Z)\eta(X)\\
&+2b\eta(X)g(AY,Z)-X(a)g(Y,Z)-X(b)\eta(Y)\eta(Z)\Big\}.
\end{align*}
That means that
\begin{align}\label{5.6}
  (\mathcal{L}_V\nabla)(Y,Z) &=-Z(a)Y-Z(b)\eta(Y)\xi-Y(a)Z-Y(b)\eta(Z)\xi\\
&-2b g(AY,Z)\xi+g(Z,Y)\nabla a+\eta(Y)\eta(Z)\nabla b.\nonumber
\end{align}
Taking the covariant differentiation of $(\mathcal{L}_V\nabla)(Y,Z)$ along any vector field $X$, we may obtain
\begin{align*}
  (&\nabla_X\mathcal{L}_V\nabla)(Y,Z) \\
  &=-g(Z,\nabla_X\nabla a)Y-g(Z,\nabla_X\nabla b)\eta(Y)\xi-Z(b)g(AX,Y)\xi\\
  &-Z(b)\eta(Y)AX-g(Y,\nabla_X\nabla a)Z-g(Y,\nabla_X\nabla b)\eta(Z)\xi-Y(b)g(AX,Z)\xi\\
  &-Y(b)\eta(Z)AX-2X(b) g(AY,Z)\xi-2b g((\nabla_XA)Y,Z)\xi-2b g(AY,Z)AX\\
&+g(Z,Y)\nabla_X\nabla a+g(AX,Y)\eta(Z)\nabla b+\eta(Y)g(AX,Z)\nabla b+\eta(Y)\eta(Z)\nabla_X\nabla b.
\end{align*}
Thus by virtue of \eqref{3.8} we have
\begin{equation*}\begin{aligned}
(&\mathcal{L}_VR)(X,Y)Z \\
=&(\nabla_X\mathcal{L}_V\nabla)(Y,Z)-(\nabla_Y\mathcal{L}_V\nabla)(X,Z)\\
=&-g(Z,\nabla_X\nabla a)Y-g(Z,\nabla_X\nabla b)\eta(Y)\xi-Z(b)\eta(Y)AX-Y(b)g(AX,Z)\xi\\
  &-Y(b)\eta(Z)AX-2X(b) g(AY,Z)\xi-2b g((\nabla_XA)Y,Z)\xi-2b g(AY,Z)AX\\
&+g(Z,Y)\nabla_X\nabla a+\eta(Y)g(AX,Z)\nabla b+\eta(Y)\eta(Z)\nabla_X\nabla b\\
&-\Big[-g(Z,\nabla_Y\nabla a)X-g(Z,\nabla_Y\nabla b)\eta(X)\xi-Z(b)\eta(X)AY-X(b)g(AY,Z)\xi\\
  &-X(b)\eta(Z)AY-2Y(b) g(AX,Z)\xi-2b g((\nabla_YA)X,Z)\xi-2b g(AX,Z)AY\\
&+g(Z,X)\nabla_Y\nabla a+\eta(X)g(AY,Z)\nabla b+\eta(X)\eta(Z)\nabla_Y\nabla b\Big]
\end{aligned}\end{equation*}
since $g(X,\nabla_Y\nabla\zeta)=g(Y,\nabla_X\nabla\zeta)$ for any function $\zeta$ and vector fields $X,Y$ followed from Poincare lemma.

By contracting over $X$ in the previous formula, we have
\begin{align}\label{5.7}
  (&\mathcal{L}_VRic)(Y,Z) \\
=&g(Z,\nabla_Y\nabla a)-g(Z,\nabla_\xi\nabla b)\eta(Y)-2fZ(b)\eta(Y)\nonumber\\
  &-2fY(b)\eta(Z)-2\xi(b) g(AY,Z)-2b g((\nabla_\xi A)Y,Z)-4fb g(AY,Z\nonumber)\\
&+g(Z,Y)\Delta a+\eta(Y)g(A\nabla b,Z)+\eta(Y)\eta(Z)\Delta b\nonumber\\
&-\Big[-g(Z,\nabla_Y\nabla b)-AY(b)\eta(Z)+\eta(Z)g(\nabla_Y\nabla b,\xi)\Big].\nonumber
\end{align}

On the other hand, since $QX=a X+b\eta(X)\xi$ and $r=3a+b$, the equation \eqref{4.6} is expressed as
\begin{align}\label{5.8}
R(X,Y)Z=&b g(Y,Z)\eta(X)\xi-b g(X,Z)\eta(Y)\xi+b\eta(Y)\eta(Z)X-b\eta(X)\eta(Z)Y\\
&+\frac{a-b}{2}[g(Y,Z)X-g(X,Z)Y].\nonumber
\end{align}
Thus we get the Lie derivative of $R(X,Y)Z$ along $V$  from \eqref{5.8}
\begin{align}\label{5.9}
(&\mathcal{L}_VR)(X,Y)Z\\
=&V(b) g(Y,Z)\eta(X)\xi+b(\mathcal{L}_Vg)(Y,Z)\eta(X)\xi\nonumber\\
&+b g(Y,Z)(\mathcal{L}_V\eta)(X)\xi+b g(Y,Z)\eta(X)\mathcal{L}_V\xi\nonumber\\
&-[V(b) g(X,Z)\eta(Y)\xi+b(\mathcal{L}_Vg)(X,Z)\eta(Y)\xi\nonumber\\
&+b g(X,Z)(\mathcal{L}_V\eta)(Y)\xi+b g(X,Z)\eta(Y)\mathcal{L}_V\xi]\nonumber\\
&+V(b)\eta(Y)\eta(Z)X+b(\mathcal{L}_V\eta)(Y)\eta(Z)X+b\eta(Y)(\mathcal{L}_V\eta)(Z)X\nonumber\\
&-[V(b)\eta(X)\eta(Z)Y+b(\mathcal{L}_V\eta)(X)\eta(Z)Y+b\eta(X)(\mathcal{L}_V\eta)(Z)Y]\nonumber\\
&+\frac{V(a-b)}{2}[g(Y,Z)X-g(X,Z)Y]+\frac{a-b}{2}[(\mathcal{L}_Vg)(Y,Z)X-(\mathcal{L}_Vg)(X,Z)Y].\nonumber
\end{align}
Contracting over $X$ in \eqref{5.9} gives
\begin{align}\label{5.10}
(&\mathcal{L}_VRic)(Y,Z)\\
=&V(b) g(Y,Z)+\ b(\mathcal{L}_Vg)(Y,Z)+\ b g(Y,Z)(\mathcal{L}_V\eta)(\xi)\nonumber\\
&+\ b g(Y,Z)g(\mathcal{L}_V\xi,\xi)-[V(\ b)\eta(Z)\eta(Y)+\ b(\mathcal{L}_Vg)(\xi,Z)\eta(Y)\nonumber\\
&+ b\eta(Z)(\mathcal{L}_V\eta)(Y)+ b\eta(Y)g(\mathcal{L}_V\xi,Z)]\nonumber\\
&+3V( b)\eta(Y)\eta(Z)+3 b(\mathcal{L}_V\eta)(Y)\eta(Z)+3 b\eta(Y)(\mathcal{L}_V\eta)(Z)\nonumber\\
&-[V( b)\eta(Y)\eta(Z)+ b(\mathcal{L}_V\eta)(Y)\eta(Z)+ b\eta(Y)(\mathcal{L}_V\eta)(Z)]\nonumber\\
&+\frac{V(a- b)}{2}[2g(Y,Z)]+\frac{a- b}{2}[2(\mathcal{L}_Vg)(Y,Z)]\nonumber\\
=& b g(Y,Z)g(\mathcal{L}_V\xi,\xi)-[ b(\mathcal{L}_Vg)(\xi,Z)\eta(Y)+ b\eta(Y)g(\mathcal{L}_V\xi,Z)]\nonumber\\
&+V( b)\eta(Y)\eta(Z)+ b(\mathcal{L}_V\eta)(Y)\eta(Z)+2 b\eta(Y)(\mathcal{L}_V\eta)(Z)\nonumber\\
&+V( a)g(Y,Z)+ a(\mathcal{L}_Vg)(Y,Z)\nonumber\\
=& b g(Y,Z)g(\mathcal{L}_V\xi,\xi)-[2 b(\lambda- a- b)\eta(Z)\eta(Y)+ b\eta(Y)g(\mathcal{L}_V\xi,Z)]\nonumber\\
&+V( b)\eta(Y)\eta(Z)+ b[g(AY,V)+\eta(\nabla_YV)]\eta(Z)\nonumber\\
&+2 b\eta(Y)[g(AZ,V)+\eta(\nabla_ZV)]+V( a)g(Y,Z)\nonumber\\
&+ a\Big(2\lambda g(Y,Z)-2 a g(Y,Z)-2 b\eta(Y)\eta(Z)\Big).\nonumber
\end{align}
Finally, by comparing \eqref{5.10} with \eqref{5.7} we get
\begin{equation*}\begin{aligned}
 &g(Z,\nabla_Y\nabla a)-g(Z,\nabla_\xi\nabla b)\eta(Y)-2fZ( b)\eta(Y)\\
  &-2fY( b)\eta(Z)-2\xi( b) g(AY,Z)-2 b g((\nabla_\xi A)Y,Z)-4f b g(AY,Z)\\
&+g(Z,Y)\Delta a+\eta(Y)g(A\nabla b,Z)+\eta(Y)\eta(Z)\Delta b\\
&-\Big[-g(Z,\nabla_Y\nabla b)-AY( b)\eta(Z)+\eta(Z)g(\nabla_Y\nabla b,\xi)\Big]\\
=& b g(Y,Z)g(\mathcal{L}_V\xi,\xi)-[2 b(\lambda- a- b)\eta(Z)\eta(Y)+ b\eta(Y)g(\mathcal{L}_V\xi,Z)]\\
&+V( b)\eta(Y)\eta(Z)+ b[g(AY,V)+\eta(\nabla_YV)]\eta(Z)\\
&+2 b\eta(Y)[g(AZ,V)+\eta(\nabla_ZV)]+V( a)g(Y,Z)\\
&+ a(2\lambda g(Y,Z)-2 a g(Y,Z)-2 b\eta(Y)\eta(Z)).
\end{aligned}\end{equation*}
Replacing $Y$ and $Z$ by $\phi Y$ and $\phi Z$ respectively leads to
\begin{align}\label{5.11}
 &g(\phi Z,\nabla_{\phi Y}\nabla a)-2\xi(b) g(A\phi Y,\phi Z)-2 b g((\nabla_\xi A)\phi Y,\phi Z)\\
 &-4f b g(A\phi Y,\phi Z)+g(\phi Z,\phi Y)\Delta a-\Big[-g(\phi Z,\nabla_{\phi Y}\nabla b)\Big]\nonumber\\
=& b g(\phi Y,\phi Z)g(\mathcal{L}_V\xi,\xi)+V( a)g(\phi Y,\phi Z)+2 a(\lambda - a)g(\phi Y,\phi Z).\nonumber
\end{align}
Because $\sum_{i=1}^3(\mathcal{L}_V(e_i,e_i))=0$ (see \cite[Eq.(9)]{G}), by \eqref{5.6} we obtain the gradient field
\begin{equation}\label{5.12}
  \nabla( a+ b)=2[\xi( b)+2f b]\xi.
\end{equation}
Therefore the formula \eqref{5.11} can be simplified as
\begin{align}\label{5.13}
-2& b g((\nabla_\xi A)\phi Y,\phi Z)+g(\phi Z,\phi Y)\Delta a\\
=& b g(\phi Y,\phi Z)g(\mathcal{L}_V\xi,\xi)+V( a)g(\phi Y,\phi Z)+2 a(\lambda - a)g(\phi Y,\phi Z).\nonumber
\end{align}

Moreover, using \eqref{2.5} and \eqref{5.3} we compute
\begin{align*}
  (\nabla_\xi A)\phi Y=&-\xi(f)\phi^3Y+\phi^2(\nabla_\xi h)Y\\
  =&\xi(f)\phi Y+\Big(\frac{ a+ b}{2}+\xi(f)+f^2\Big)\phi Y+2fhY+\phi h^2Y\\
  =&\Big(\frac{ a+ b}{2}+2\xi(f)+f^2\Big)\phi Y+2fhY+\phi h^2Y.
\end{align*}
Substituting this into \eqref{5.13} yields
\begin{align}\label{5.14}
-2& b g(2fhY+\phi h^2Y,\phi Z)\\
=&\Big\{ b g(\mathcal{L}_V\xi,\xi)+2b\Big[\frac{ a+ b}{2}+2\xi(f)+f^2\Big]\nonumber\\
&-\Delta a+V( a)+2 a(\lambda - a)\Big\}g(\phi Y,\phi Z).\nonumber
\end{align}
Replacing $Z$ by $\phi Y$ in \eqref{5.14} gives
\begin{equation}\label{5.15}
 0= b g(2fhY+\phi h^2Y, \phi^2Y )=- b g(2fhY+\phi h^2Y, Y).
\end{equation}

Next we divide into two cases.

{\bf Case I:} $h=0$. Then the Eq.\eqref{2.8} and
Proposition \ref{p2} imply $\xi(f)=0$, i.e., $f$ is constant as $\nabla f=\xi(f)\xi$
 followed from $df\wedge\eta=0$, so $M$ is an $\alpha$-cosymplectic manifold.

 {\bf Case II:} $h\neq0$. Suppose that $Y=e$ is an unit eigenvector corresponding to the
 nonzero eigenvalue $\lambda'$ of $h$, then we may obtain from \eqref{5.15} that $f b=0$.
 If the function $b$ is not zero, there is an open neighborhood $\mathcal{U}$
 such that $b|_\mathcal{U}\neq0$, so $f|_\mathcal{U}=0.$ By Proposition \ref{p2} it implies $(a+b)|_\mathcal{U}=0$.
 Notice that $\mathrm{trace}(h^2)=-(a+b)-2(\xi f+f^2)$, so we have $\mathrm{trace}(h^2)|_\mathcal{U}=0$, i.e., $h|_\mathcal{U}=0.$
It comes to a contradiction, so $b=0$ and $ a$ is constant by \eqref{5.12}.
Hence by \eqref{1.3} we obtain $Ric=ag$, where $a=-2\xi(f)-2f^2.$
Moreover, we know $\lambda= a$ or $ a=0$ by \eqref{5.13}.
 Finally we show $ a\neq0$. In fact, if $ a=0$ then $\xi(f)=-f^2$. Substituting this into \eqref{2.8} we find $\mathrm{trace}(h^2)=-2(\xi(f)+f^2)=0$, which is impossible because $h\neq0$.

  Thus we obtain  from \eqref{5.8}
$$R(X,Y)Z=\frac{\lambda}{2}(g(Y,Z)X-g(X,Z)Y).$$
Summing up the above two cases we finish the proof Theorem \ref{T3}.
\section*{Acknowledgement}
The author would like to thank the referee for the comments and valuable suggestions.

%
%

%
%

%

%
%
\end{document}